\documentclass[letterpaper, 10pt, conference]{ieeeconf}

\IEEEoverridecommandlockouts                              	%
\usepackage{amssymb,amsmath,amsfonts}
\usepackage{algorithm,algorithmic}
\usepackage{blkarray}
\usepackage{bbm}
\usepackage{bm}
\usepackage[makeroom]{cancel}
\usepackage{cases}
\usepackage{cite}
\usepackage{changebar}
\usepackage[usenames,dvipsnames]{color}
\usepackage{dsfont}
\usepackage{epsfig}
\usepackage[T1]{fontenc}
\usepackage{float}
\usepackage{flushend}
\usepackage{graphicx}
\usepackage{hyperref} 
\usepackage[latin9]{inputenc}
\usepackage{mathtools}
\usepackage{mathrsfs}
\usepackage{setspace}
\usepackage{soul} 
\usepackage{subfigure}
\usepackage{stfloats}
\usepackage{tikz}
\usepackage{url}
\usepackage{verbatim}
\usepackage{xspace}

\usetikzlibrary{automata,chains,arrows}
\usetikzlibrary{arrows.meta}

\floatstyle{ruled}
\newfloat{model}{H}{mod}
\floatname{model}{\footnotesize Model}
\newfloat{notatio}{H}{not}
\floatname{notatio}{\footnotesize Notation}

\newtheorem{remark}{\bfseries Remark}
\newtheorem{theorem}{\bfseries Theorem}
\newtheorem{lemma}{\bfseries Lemma}
\newtheorem{assumption}{\bfseries Assumption}

\newcommand{\T}{^{\mbox{\tiny T}}}

\let\mathbb=\mathds %

\floatstyle{ruled}
\newfloat{model}{H}{mod}
\floatname{model}{\footnotesize Model}
\newfloat{notatio}{H}{not}
\floatname{notatio}{\footnotesize Notation}

\newenvironment{list4}{
	\begin{list}{$\bullet$}{%
			\setlength{\itemsep}{0.05cm}
			\setlength{\labelsep}{0.2cm}
			\setlength{\labelwidth}{0.3cm}
			\setlength{\parsep}{0in} 
			\setlength{\parskip}{0in}
			\setlength{\topsep}{0in} 
			\setlength{\partopsep}{0in}
			\setlength{\leftmargin}{0.19in}}}
	{\end{list}}

\begin{document}
	\title{\fontsize{21}{24}\selectfont An Asynchronous Approximate Distributed \\ Alternating Direction Method of Multipliers in Digraphs}
	
	\author{Wei Jiang, %
		Andreas Grammenos, %
		Evangelia Kalyvianaki, %
		and Themistoklis Charalambous%
		\thanks{W. Jiang and T. Charalambous are with the Department of Electrical Engineering and Automation, School of Electrical Engineering, Aalto University, Espoo, Finland.
			Emails: {\tt \{name.surname\}@aalto.fi}.}
		\thanks{A. Grammenos is with the Department of Computer Science and Technology, University of Cambridge, Cambridge, and the Alan Turing Institute, London, UK. Email: {\tt ag926@cl.cam.ac.uk}.}
		\thanks{E. Kalyvianaki is with the Department of Computer Science and Technology, University of Cambridge, Cambridge, UK. Email: {\tt ek264@cl.cam.ac.uk}.}
	}

	\markboth{IEEE TRANSACTIONS ON Automatic Control}%
	{Shell \MakeLowercase{\textit{et al.}}: Bare Demo of IEEEtran.cls for IEEE Journals}
	
	\maketitle
	\thispagestyle{empty}
	\pagestyle{empty}
	
	\begin{abstract}
		In this work, we consider the asynchronous distributed optimization problem in which each node has its own convex cost function and can communicate directly only with its neighbors, as determined by a directed communication topology (directed graph or digraph). First, we reformulate the optimization problem so that Alternating Direction Method of Multipliers (ADMM) can be utilized. Then, we propose an algorithm, herein called Asynchronous Approximate Distributed Alternating Direction Method of Multipliers (AsyAD-ADMM), using finite-time asynchronous approximate ratio consensus, to solve the multi-node convex optimization problem, in which every node performs iterative computations and exchanges information with its neighbors asynchronously. More specifically, at every iteration of AsyAD-ADMM, each node solves a local convex optimization problem for one of the primal variables and utilizes a finite-time asynchronous approximate consensus protocol to obtain the value of the other variable which is close to the optimal value, since the cost function for the second primal variable is not decomposable. If the individual cost functions are convex but not necessarily differentiable, the proposed algorithm converges at a rate of $O(1/k)$, where $k$ is the iteration counter. The efficacy of AsyAD-ADMM is exemplified via a proof-of-concept distributed least-square optimization problem with different performance-influencing factors investigated.
	\end{abstract}
	
	
	\section{Introduction}
	\label{sec:introduction}

	In this paper, we are interested in developing a modified ADMM algorithm for solving large-scale optimization problems in digraphs. In that direction, there are two main communication topologies considered: \emph{(i)} master-workers communication topology \cite{zhang2014asynchronous,chang2016asynchronous,li2020synchronous} and \emph{(ii)} multi-node communication topology~\cite{6736937,9146334,khatana2020d,wei_themis_2021}. When the ADMM has a master-worker communication topology, the worker nodes optimize their local objectives and communicate their local variables to the master node which updates the global optimization variable and send it back to the workers. When the ADMM has no master node, the optimization problem is solved over a network of nodes. Herein, we focus on the ADMM realized on multi-node communication topologies.
	
	There are several real-life applications in which synchronous operation is infeasible or costly and complicated. As a result, asynchronous information exchange is necessitated. For example, resource allocation in data centers gives rise to large-scale problems and networks, which naturally call for asynchronous solutions. Recently,  different asynchronous distributed ADMM versions are proposed under different assumptions and communication topologies. For example, the version in~\cite{zhang2014asynchronous,chang2016asynchronous} is based on the master-worker architecture under a star topology. Even though worker updates are parallel and distributed, the master update is centralized as it collects information from workers and broadcasts the updated information to workers for their following iteration. Authors in~\cite{6736937,9146334} proposed another version with the almost surely convergence by using information from a randomly selected subset of nodes for ADMM update under the positive possibility selection assumption.  Then, authors in~\cite{li2020synchronous} combined the above two architecture and designed a master-local processor-link processor architecture with the random node selection. It is worth noting that all the above asynchronous versions are only available for undirected communication graphs, i.e., assuming that every communication link is bidirectional. Very recently, distributed ADMM versions for digraphs are proposed in~\cite{khatana2020d,wei_themis_2021}. However, they are synchronous, which is also the motivation for this work. %
	
	In this paper, we propose AsyAD-ADMM, an asynchronous distributed ADMM algorithm for digraphs.
	First, the optimization problem is restructured such that at every optimization step, each node can solve individually a local convex optimization problem for one of the primal variables, while for the other primal variable the communication among nodes is required, since the cost function is not decomposable. Then, we find the solution for the primal variable whose cost function is non-decomposable by means of a finite-time asynchronous approximate consensus algorithm. More specifically, we adopt the algorithm introduced in~\cite[Algorithm 2]{grammenos2021cpu}. The algorithm relies on asynchronous \texorpdfstring{$\min-$consensus}{min-consensus} and \texorpdfstring{$\max-$consensus}{max-consensus} protocols to compute the approximate average consensus value considering different communication delays over a \emph{finite}  number of steps.
	Unlike \cite{grammenos2021cpu}, the consensus algorithm should be repeated at every optimization step, which is also asynchronous. Towards this end, we propose a mechanism in which every next optimization step is initiated after the previous one has finished.
	The main contributions of the paper are as follows.
	\begin{list4}
		\item We propose an asynchronous distributed ADMM algorithm which can be applied to digraphs.
		\item The convergence analysis shows AsyAD-ADMM converges at a rate of $O(1/k)$, where $k$ is the iteration counter. Factors influencing the algorithm performance are also investigated.
	\end{list4}

	\section{Preliminaries}\label{sec:preliminaries}
	
	\subsection{Notation and graph theory}
	
	The set of real (non-negative integer, positive integer) numbers is denoted by $\mathds{R} $ ($\mathds{Z}, \mathds{N} $)  and   $\mathds{R}^n_+$ denotes the non-negative orthant of the $n$-dimensional real space $\mathds{R}^n$. 
	$A\T$ denotes the transpose of matrix $A$. 
	For $A\in \mathbb{R}^{n\times n}$, $a_{ij}$ denotes the entry in row $i$ and column $j$.
	By $\mathbb{1}$ we denote the all-ones vector and by $I$ we denote the identity matrix (of appropriate dimensions). 
	$ \|\cdot\| $ denotes the 2-norm.
	
	In multi-node systems with fixed communication links (edges), the exchange of information between nodes can be conveniently captured by a graph  $\mathcal{G}(\mathcal{V}, \mathcal{E})$ of order $n$ $(n \geq 2)$, where $\mathcal{V} = \{v_1,v_2,\ldots,v_n\}$ is the set of nodes and $\mathcal{E} \subseteq \mathcal{V} \times \mathcal{V}$ is the set of edges. A directed edge from node $v_i$ to node $v_j$ is denoted by $\varepsilon_{ji} = (v_j, v_i)\in \mathcal{E}$ and represents a communication link that allows node $v_j$ to receive information from node $v_i$. A graph is said to be undirected if and only if $\varepsilon_{ji} \in \mathcal{E}$ implies $\varepsilon_{ij}  \in \mathcal{E}$. A digraph is called \emph{strongly} connected if there exists a path from each vertex $v_i$ of the graph to each vertex $v_j$ ($v_j \neq v_i$). In other words, for any $v_j, v_i \in \mathcal{V}$, $v_j\neq v_i$, one can find a sequence of nodes $v_i = v_{l_1}$, $v_{l_2}$, $v_{l_3}$, $\ldots$, $v_{l_m}=v_j$ such that link $(v_{l_{s+1}}, v_{l_{s}}) \in \mathcal{E}$ for all $s = 1, 2, \ldots, m-1$. The diameter $D$ of a graph is the longest shortest path between any two nodes in the network.
	
	All nodes that can transmit information to node $v_j$ directly are said to be in-neighbors of node $v_j$ and belong to the set $\mathcal{N}^{-}_j=\{ v_i \in \mathcal{V} \; | \; \varepsilon_{ji} \in \mathcal{E} \}$. The cardinality of $\mathcal{N}^{-}_j$, is called the \emph{in-degree} of $v_j$ and is denoted by $\mathcal{D}^{-}_{j}=\left| \mathcal{N}^{-}_j \right|$. The nodes that receive information from node $v_j$ belong to the set of out-neighbors of node $v_j$, denoted by $\mathcal{N}^{+}_j=\{ v_l \in \mathcal{V} \; | \; \varepsilon_{lj} \in \mathcal{E} \}$. The cardinality of $\mathcal{N}^{+}_j$, is called the \emph{out-degree} of $v_j$ and is denoted by $\mathcal{D}^{+}_{j}= \left| \mathcal{N}^{+}_j \right|$.

	In the type of algorithms we consider, we associate a positive weight $p_{ji}$ for each edge $\varepsilon_{ji} \in \mathcal{E} \cup \{ (v_j, v_j) \; | \: v_j \in \mathcal{V} \}$. The nonnegative matrix $P = [p_{ji} ] \in \mathbb{R}_{+}^{n\times n}$ 
	is a weighted adjacency matrix 
	that has zero entries at locations that do not correspond to directed edges (or self-edges) in the graph.

	\subsection{Standard ADMM Algorithm}
	
	The standard ADMM algorithm solves the problem:%
	\begin{equation}\label{standard_ADMM_obj}
	\begin{aligned}
	\min \, &f(x) + g(z) \\
	\text{s.t.} \, &Ax + Bz = c
	\end{aligned}
	\end{equation}
	for variables $ x \in \mathbb{R}^{p}, z \in \mathbb{R}^{m} $ with matrices $ A \in \mathbb{R}^{q\times p}, B \in \mathbb{R}^{q\times m} $ and vector $ c \in \mathbb{R}^{q}$ ($p, m, q\in \mathbb{N}$).  The augmented Lagrangian is
	\begin{equation}\label{augmented_Lagrangian}
	\begin{aligned}
	L_{\rho}(x,z,\lambda) =& f(x) + g(z)+ \lambda^{T}(Ax + Bz - c) \\&+ \frac{\rho}{2}\|Ax + Bz - c\|^{2},
	\end{aligned}
	\end{equation}
	where $ \lambda $ is the Lagrange multiplier and $ \rho >0 $ is a penalty parameter. In ADMM, the primary variables $ x, z $ and the Lagrange multiplier $ \lambda $ are updated as follows: starting from some initial vector $\begin{bmatrix}x^0 & z^0 & \lambda^0 \end{bmatrix}\T $, at each optimization iteration $ k $, 
	\begin{align}
	x^{k+1}=&  \operatorname*{argmin}_x L_{\rho}(x,z^k,\lambda^k), \label{admm_x}\\
	z^{k+1} =& \operatorname*{argmin}_z L_{\rho}(x^{k+1},z,\lambda^k), \label{admm_z}\\
	\lambda^{k+1} =& \lambda^k + \rho (Ax^{k+1} + Bz^{k+1} - c) \label{admm_lamda}.
	\end{align}
	The step-size in the Lagrange multiplier update is the same as the augmented Lagrangian function parameter $ \rho $.

	\section{Problem Formulation}\label{sec:formulation}
	
	In this work, we consider a strongly connected digraph $\mathcal{G}(\mathcal{V}, \mathcal{E})$ in which each node $v_i \in\mathcal{V}$ is endowed with a scalar cost function $f_i : \mathbb{R}^p \mapsto \mathbb{R}$ assumed to be known to the node only. We assume that each node $v_i$ has knowledge of the number of its out-going links, $\mathcal{D}^{+}_{i}$, and has access to local information only via its communication with the in-neighboring nodes, $\mathcal{N}^{-}_{i}$. The only global information available to all the nodes in the network is given in Assumption~\ref{assup_graph}.
	
	\begin{assumption}\label{assup_graph}
		The diameter of the network $D$, or an upper bound, is known to all  nodes.
	\end{assumption}
	
	While Assumption~\ref{assup_graph} is limiting, there exist distributed methods for extracting such information; see, e.g., \cite{2012:Allerton_Shames,ThemisTCNS:2016}.
	
	The objective is to design a discrete-time coordination algorithm that allows every node $v_i$ in a digraph to distributively and asynchronously solve the following global optimization problem:
	\begin{equation}\label{problem_initial}
	\operatorname*{min}_{x\in \mathbb{R}^{p}}  \sum_{i=1}^{n} f_i(x),
	\end{equation}
	where $ x\in \mathbb{R}^{p} $ is a global optimization variable (or a common decision variable). In order to solve problem and  to enjoy the structure ADMM scheme at the same time, a separate decision variable $ x_i $ for node $v_i$ is introduced and the constraint $ \| x_i - x_j\| \le \epsilon, \epsilon > 0, \forall v_i, v_j\in\mathcal{V} $ is imposed to allow the asynchronous distributed ADMM framework. Thus, problem~\eqref{problem_initial} is reformulated as
	\begin{align}
	\min_{x_i, i=1,\ldots, n} \, &\sum_{i=1}^{n} f_i(x_i), \label{problem_reformulated}\\
	\text{s.t.}  \, &\| x_i - x_j\| \le \epsilon, \forall v_i, v_j\in\mathcal{V}, \label{constaint1}
	\end{align}
	where $ \epsilon $ is a predefined error tolerance.
	Define a closed nonempty convex set $ \mathcal{C} $ as
	\begin{equation}\label{setC}
	\mathcal{C} = \left\{\begin{bmatrix} x_1\T & x_2\T & \ldots & x_n\T\end{bmatrix}\T \in \mathbb{R}^{np} \, :\,  \| x_i - x_j\| \le \epsilon \right\} .
	\end{equation}
	By denoting $ X \coloneqq \begin{bmatrix} x_1\T & x_2\T & \ldots & x_n\T\end{bmatrix}\T $ and making variable $ z \in \mathbb{R}^{np} $ as a copy of vector $ X $, constraint~\eqref{constaint1} related to problem~\eqref{problem_reformulated} becomes
	\begin{equation}\label{problem_reformulated2}
	\text{s.t.}  \,  X = z, \, z \in \mathcal{C}. 
	\end{equation}
	Then, define $  g(z) $ as the indicator function of set $ \mathcal{C}  $ as 
	\begin{equation}\label{indicator_function}
	g(z) =
	\left\{ 
	\begin{array}{l}
	\begin{aligned}
	&0,\quad \text{if} \, z \in \mathcal{C},\\
	&\infty, \, \text{otherwise}.
	\end{aligned}
	\end{array}
	\right. 
	\end{equation}
	Finally, problem~\eqref{problem_reformulated} with constraint~\eqref{problem_reformulated2} is transformed to 
	\begin{equation}\label{objective_function}
	\begin{aligned}
	\min_{z, x_i, i=1,\ldots, n} \, & \left\{ \sum_{i=1}^{n} f_i(x_i) + g(z) \right\}, \\
	\text{s.t.} \, & X - z = 0. 
	\end{aligned}
	\end{equation}

	For notational convenience, denote $ F(X) \coloneqq \sum_{i=1}^{n} f_i(x_i) $. Thus, denote the Lagrangian function as
	\begin{equation}\label{Lagrangian}
	L(X,z,\lambda) = F(X) + g(z) + \lambda\T(X - z ),
	\end{equation}
	where $ \lambda  \in  \mathbb{R}^{np} $ is the Lagrange multiplier. The following standard assumptions are required for problem~\eqref{objective_function}.
	\begin{assumption}\label{assup_convex}
		Each cost function $ f_i : \mathbb{R}^{p} \rightarrow \mathbb{R} \cup \{+\infty\} $ is closed, proper and convex.
	\end{assumption}
	\begin{assumption}\label{assup_saddel_point}
		The Lagrangian $ L(X,z,\lambda) $ has a saddle point, i.e., there exists a solution $ (X^{*},z^{*},\lambda^{*}) $, for which 
		\begin{equation}\label{saddle_point}
		L(X^{*},z^{*}, \lambda)\le L (X^{*},z^{*},\lambda^{*})\le L(X,z,\lambda^{*}),
		\end{equation}
		holds for all $ X $ in $ \mathbb{R}^{np} $, $ z $ in $ \mathbb{R}^{np} $ and $ \lambda $ in $ \mathbb{R}^{np} $.
	\end{assumption}

	Assumption~\ref{assup_convex} allows $ f_i $ to be non-differentiable~\cite{boyd2011distributed}. By Assumptions~\ref{assup_convex}-\ref{assup_saddel_point} and based on the definition of $ g(z) $ in~\eqref{indicator_function}, $ L(X,z,\lambda^{*}) $ is convex in $ (X,z) $ and $ (X^{*},z^{*}) $ is a solution to problem~\eqref{objective_function}~\cite{boyd2011distributed,wei2012distributed}.
	
	At iteration $ k $, the corresponding augmented Lagrangian of optimization problem~\eqref{objective_function} is written as
	\begin{align}\label{augmented_Lagrangian2}
	L_{\rho}&(X^{k},z^{k},\lambda^{k}) \\
	=& \sum_{i=1}^{n} f_i(x_i^{k}) + g(z^{k}) + {\lambda^{k}}\T(X^{k} - z^{k} )+ \frac{\rho}{2}\|X^{k} - z^{k} \|^{2} \nonumber \\
	=&  \sum_{i=1}^{n}\left( f_i(x_i^{k}) + {\lambda_i^{k}}\T(x_i^{k} - z_i^{k}) + \frac{\rho}{2}\|x_i^{k} - z_i^{k} \|^{2} \right) +g(z^{k}) , \nonumber
	\end{align}
	where $ z_i\in \mathbb{R}^{p}  $ is the $ i$-th element of vector $ z $. By ignoring terms which are independent of the minimization variables (i.e., $ x_i, z $), for each node $ v_i $, the standard ADMM updates \eqref{admm_x}-\eqref{admm_lamda} change to the following format:
	\begin{align}
	x_i^{k+1} =&  \operatorname*{argmin}_{x_i} f_i(x_i) + {\lambda_i^{k}}\T x_i + \frac{\rho}{2}\|x_i - z_i^{k} \|^{2}, \label{dadmm_x}\\
	z^{k+1} =& \operatorname*{argmin}_z g(z) + {\lambda^{k}}\T(X^{k+1} - z )  + \frac{\rho}{2}\|X^{k+1} - z \|^{2}\nonumber \\
	=& \operatorname*{argmin}_z g(z) +\frac{\rho}{2}\|X^{k+1} - z + \frac{1}{\rho}\lambda^{k} \|^{2}, \label{dadmm_z}\\
	\lambda_i^{k+1} =& \lambda_i^{k} + \rho (x_i^{k+1} - z_i^{k+1}) \label{dadmm_lamda},
	\end{align}
	where the last term in~\eqref{dadmm_z} comes from the identity $ 2a^Tb + b^2=(a+b)^2-a^2 $ with $ a = \lambda^{k}/\rho $ and $ b = X^{k+1} - z $.

	Update~\eqref{dadmm_x} for $ x_i^{k+1} $ can be solved by a classical method, e.g., the proximity operator~\cite[Section 4]{boyd2011distributed}. Update~\eqref{dadmm_lamda} for the dual variable $ \lambda_i^{k+1} $  can be implemented trivially by node $ v_i $. 
	Note that both updates can be done independently and in parallel by node $ v_i $.

	Since $ g $ in~\eqref{indicator_function} is the indicator function of the closed nonempty convex set $ \mathcal{C}  $,  update~\eqref{dadmm_z} for $ z^{k+1} $ becomes 
	\begin{align}
	    z^{k+1} = \Pi_{\mathcal{C}}(X^{k+1}+ \lambda^{k}/\rho) , 
	\end{align}
	where $ \Pi_{\mathcal{C}} $ denotes the projection (in the Euclidean norm) onto $ \mathcal{C} $.
	Intuitively, from \eqref{dadmm_z} and the definition of $ g(z) $ in~\eqref{indicator_function}, one can see that the elements of $ z $ (i.e., $ z_1, z_2, \ldots, z_n $)  should go into $ \mathcal{C} $ in finite time. If not, one will have $  g(z)  = \infty $ and update~\eqref{dadmm_z} will never be finished. Then, from the definition of $ \mathcal{C} $ in~\eqref{setC}, one can see that $ z $ going into $ \mathcal{C} $ means $  \| z_i - z_j\| \le \epsilon, \forall v_i, v_j\in\mathcal{V} $. %
	
	It is worth noting that if $  z_i - z_j =0, \forall v_i, v_j\in\mathcal{V} $, it means $  z_1= z_2= \ldots = z_n $, which is in the mathematical format of consensus.
	In other words, each node $v_i\in\mathcal{V}$ can have $z_i$ reach $ \frac{1}{n} \sum_{i=1}^{n}z_i(0), z_i(0) = x_i^{k+1} + \lambda_i^{k}/\rho  $ in a finite number of steps, i.e., finite-time average consensus.
	Similarly, requiring $  \| z_i - z_j\| \le \epsilon, \forall v_i, v_j\in\mathcal{V} $ means asking nodes $v_i, v_j \in\mathcal{V}$ to have $z_i, z_j$ with $ \epsilon $ closeness to the average consensus, i.e., to have $ z_i, z_j $ enter a circle with its center at $ \frac{1}{n} \sum_{i=1}^{n} (x_i^{k+1} + \lambda_i^{k}/\rho)  $ and its radius as $ \epsilon $. Here we call it finite-time ``approximate'' ratio consensus.

	\section{Asynchronous operation and consensus algorithms}
	\label{sec:asynchrony}
	
	Before we proceed to the description of our main algorithm, we describe the form of asynchrony that the nodes experience and some consensus algorithms that are key ingredients for the operation of our proposed algorithms.
	
		\subsection{Asynchrony description}\label{preliminary_asynchrony}
	
	 Let $t(0)\in \mathbb{R}_{+}$ the time at which the iterations for the optimization start. We assume that there is a set of times $\mathcal{T}=\{t(1), t(2),t(3),\ldots\}$ at which one or more nodes transmit some value to their neighbors. In a synchronous setting, \emph{each} node $v_i$ updates and sends its information to its neighbors at discrete times in $\mathcal{T}$ and no processing or communication delays are considered. In an asynchronous setting, a message that is received at time $t(k_1)$ and processed at time $t(k_2)$, $k_2>k_1$, experiences a process delay of $t(k_2)-t(k_1)$ (or a time-index delay $k_2-k_1$). In Fig.~\ref{fig:async}, we show through a simple example how the time steps evolve for each node in the network; with $t_i(k)$ we denote the time step at which iteration $k$ takes place for node $v_{i}$.
	\begin{figure}[ht]
		\includegraphics[width=0.97\columnwidth]{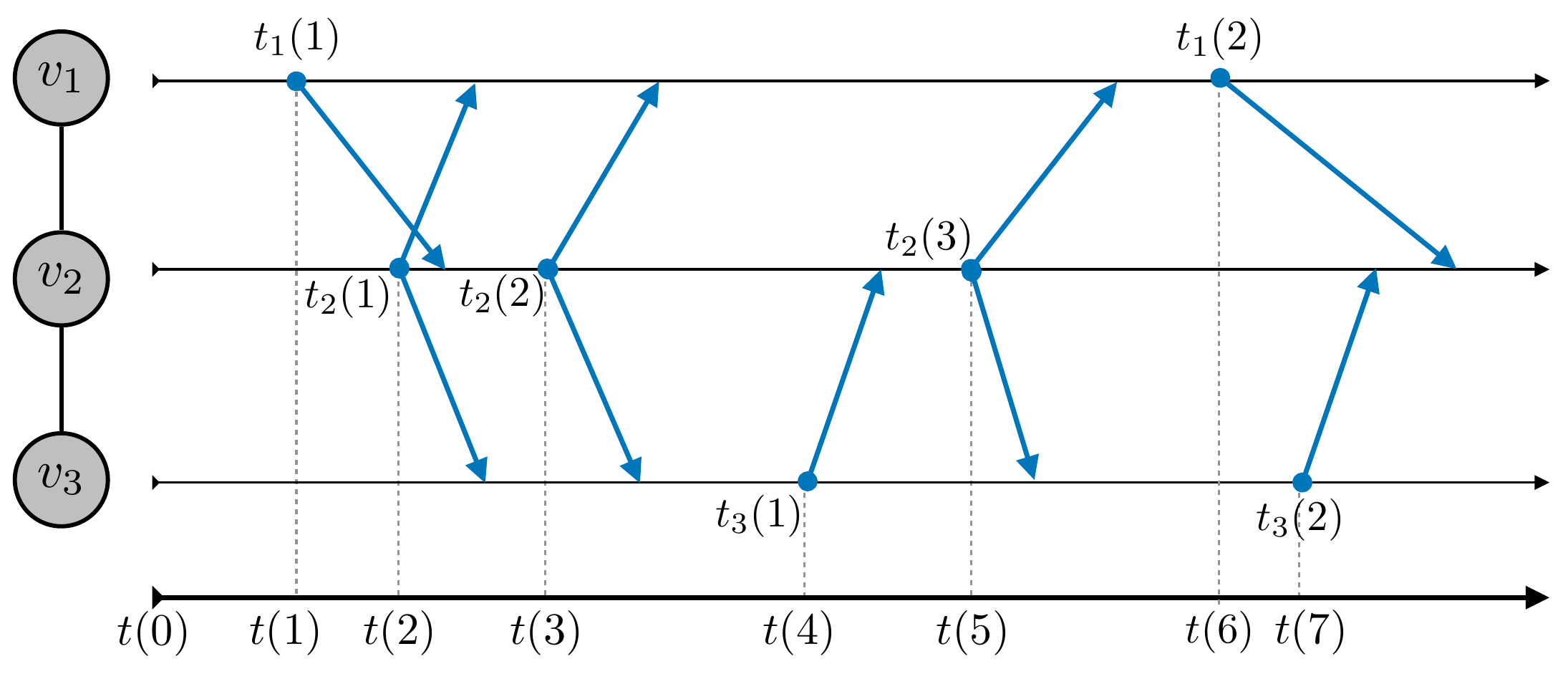}
		\caption{
			A simple example of a network consisting of 3 nodes. 
			In the timeline of each node, blue ticks indicate an iteration for node $v_{i}$ and the arrows indicate the transmissions. 
			The time in between transmissions is the processing delay, while the time from the beginning of the transmission to the end (arrow) is the transmission delay.
		}
		\label{fig:async}
	\end{figure}
	
	\noindent We index nodes' information states and any other information at time $t(k)$ by time-index $k$. Hence, we use $x_i^k \equiv x_i[k]  \in \mathbb{R}$ to denote the information state of node $i$ at time $t(k)$. Note that ${x_i^{k}}\T$ denotes (is equivalent to) $x_i[k]\T$.
	
	\begin{assumption}\label{assump_time_index_bound}
		There exists an upper bound $B, B \in \mathbb{N}$, on the time-index steps that is needed for a node to process the information received from other nodes.
	\end{assumption}
	
	Assumption~\ref{assump_time_index_bound} basically states that since the number of nodes is finite and they update their state regularly, there exists a finite number of steps $B$ before which all nodes have updated their values (and hence broadcast it to neighboring nodes). However, it is not possible for nodes to count the number of steps elapsed in the network. For this reason we make an additional assumption.
	\begin{assumption}\label{assump_time_bound}
		There exists an upper bound $T, T \in \mathbb{R}_{\geq 0}$, on the actual time in seconds that is needed for a node to process the information received from other nodes.
	\end{assumption}
	In other words, the upper bound $B$ steps in Assumption~\ref{assump_time_index_bound} is translated to an upper bound of $T$ in seconds, something that the nodes can count individually.
	
	\subsubsection{Asynchronous ratio consensus}\label{arc}
	Here, we regard the asynchronous information exchange among nodes as delayed information. Thus, we adopt a protocol developed in \cite{2014:chrisThemisTAC} where each node updates its information state $x_{j}^{k+1}$ by combining the delayed information received by its neighbors $x_{i}^s$ ($s\in \mathbb{Z}, s\leq k,~v_{i}\in \mathcal{N}^{-}_j$) using constant positive weights $p_{ji}$. Integer $\tau_{ji}^k \geq 0$ is used to represent the delay of a message sent from node $v_{i}$ to node $v_{j}$ at time instant $k$. We require that $0\leq \tau_{ji}^k \leq \bar{\tau}_{ji} \leq \bar{\tau}$ for all $k\geq 0$ for some finite $\bar{\tau} = \max\{ \bar{\tau}_{ji} \}$, $\bar{\tau} \in \mathbb{N}$ and $ 1+\bar{\tau} \le B $ where $ B $ is defined in Assumption~\ref{assump_time_index_bound}. We make the reasonable assumption that $\tau_{jj}^k=0$, $\forall v_{j} \in \mathcal{V}$, at all time instances $k$ (i.e., the own value of a node is always available without delay). Each node updates its information state according to:
	\begin{align}
	y_{j}^{k+1} =p_{jj}y_{j}^k + \sum_{v_{i} \in \mathcal{N}^{-}_j} \sum_{r=0}^{\bar{\tau} } p_{ji} y_{i}^{k-r}I_{k-r,ji}^{r}, \label{delay_rc_y_update}
	\end{align}
	\noindent for $k\geq 0$, where  $y_{j}^0 \in \mathbb{R}$ is the initial state of node $v_{j}$; $p_{ji}$ $\forall \varepsilon_{ji} \in \mathcal{E}$  form $P=[p_{ji}]$ that adheres to the graph structure, and is primitive column stochastic; and
	\begin{align}\label{eq:indicatorfunction}
	I_{k,ji}(\tau) =
	\begin{cases} 1, & \text{if $\tau_{ji}^k =\tau$,}
	\\
	0, &\text{otherwise.}
	\end{cases}
	\end{align}

	\begin{lemma}\cite[\emph{Lemma 2}]{2014:chrisThemisTAC}
		\label{our_lemma}
		Consider a strongly connected digraph $\mathcal{G}(\mathcal{V}, \mathcal{E})$. Let $y_{j}^k$ and $w_{j}^k$ (for all $v_{j} \in \mathcal{V}$ and $k=0,1,2,\ldots$) be the result of the iterations \eqref{delay_rc_y_update} and
		\begin{align}\label{eq:ratio-delays}
		w_{j}^{k+1} &=p_{jj}w_{j}^k + \sum_{v_{i} \in \mathcal{N}^{-}_j} \sum_{r=0}^{\bar{\tau} } p_{ji} w_{i}^{k-r}I_{k-r,ji}^{r} \; ,
		\end{align}
		where $p_{lj} = \frac{1}{1 + \mathcal{D}_j^+}$ for $v_l \in \mathcal{N}_j^+ \cup \{ v_j \}$ (zeros otherwise)
		with $y^0=[{y_1^0}\T, {y_2^0}\T, \ldots, {y_{|\mathcal{V}|}^0}\T]^T$ and  $w^0=\mathbb{1}$; $I_{k,ji}$ is an indicator function that captures the bounded delay $\tau_{ji}^k$ on link $(v_{j}, v_{i})$ at iteration $k$ (as defined in \eqref{eq:indicatorfunction}, $\tau_{ji}^k \leq \bar{\tau}$). Then, the solution $\displaystyle z_{j}^k={y_{j}^k}/{w_{j}^k}$ to the average consensus problem can be asymptotically obtained as
		$
		\lim_{k\rightarrow \infty} z_{j}^k=\frac{\sum_{v_\ell \in \mathcal{V}} y_{\ell}^0}{|\mathcal{V}|} \; , \; \forall v_{j} \in \mathcal{V} \; .
		$
	\end{lemma}

	\subsubsection{Asynchronous \texorpdfstring{$\max-$consensus}{max-consensus}}\label{amc}
	
	When the updates are asynchronous, for any node $v_{j}\in \mathcal{V}$, the update rule is as follows \cite{2013:Giannini}:
	\begin{align*}
	y_j[t_j(k+1)] = \max_{v_{i}\in \mathcal{N}_j^{-}[t_j(k+1)] \cup \{v_{j}\}}\{ y_i[t_j(k)+\theta_{ij}(k)] \},
	\end{align*}
	where $y_i[t_j(k)+\theta_{ij}(k)]$ are the states of the in-neighbors $\mathcal{N}_j^{-}[t_j(k+1)]$ available at the time of the update.  Variable $\theta_{ij}(k) \in \mathbb{R}$, evaluated with respect to the update time $t_j(k)$ (defined in Sec.~\ref{preliminary_asynchrony}), is used here to express asynchronous state updates occurring at the neighbors of node $v_{j}$, between
	two consecutive updates of the state of node $v_{j}$.
	It has been shown in \cite[Lemma 5.1]{2013:Giannini} that this algorithm converges to the maximum value among all nodes in a finite number of steps $s$, $s \leq B D$, where $ D $ and $ B $ are defined in Assumptions \ref{assup_graph} and~\ref{assump_time_index_bound}, respectively.

	\section{Main results}
	\label{sec:results}
	
    Our proposed algorithm works on two levels: \emph{i)} at the optimization level in which the distributed ADMM procedure is described for updating $x_i^{k+1}$, $z_i^{k+1}$ and $\lambda_i^{k+1}$ (Algorithm~\ref{algorithm:1}), and
        \emph{ii)} at the distributed coordination level for computing $z_i^{k+1}$ in an asynchronous manner (Algorithm~\ref{Algorithm_finitetimeRCwDelays}). Algorithm~\ref{Algorithm_finitetimeRCwDelays} is executed within Algorithm~\ref{algorithm:1}.
	
	\subsection{Algorithm~\ref{algorithm:1}}
	
	Unlike asynchronous operation which is depicted in Fig.~\ref{fig:async}, in a synchronous algorithm all nodes need to agree on the update time $t(k)$, which usually requires synchronization among all nodes or the existence of a global clock.
	Whether the distributed ADMM update is synchronous or asynchronous arises not only when exchanging information for computing $ z^{k+1} $ using~\eqref{dadmm_z}, but also when they start the next optimization step (i.e., when optimization step $k$ ends and step $k+1$ begins). For now, we assume that the transition in the optimization steps is somehow synchronized. We will discuss ways to achieve it later.
	
	Our AsyAD-ADMM algorithm to solve optimization problem~\eqref{objective_function} is summarized in Algorithm ~\ref{algorithm:1}. In Algorithm~\ref{algorithm:1}, at every optimization step, each node $v_{i}$ does the following:
	\begin{list4}
	    \item It computes $x_i^{k+1}$ using Eq.~\eqref{dadmm_x} locally.
	    \item It computes $z_i^{k+1}$ via Algorithm~\ref{Algorithm_finitetimeRCwDelays} in an asynchronous fashion and in a finite number of steps, since Algorithm~\ref{Algorithm_finitetimeRCwDelays} deploys a termination mechanism.
	    \item It computes $\lambda_i^{k+1}$ using Eq.~\eqref{dadmm_lamda} and the values of $x_i^{k+1}$ and $z_i^{k+1}$. 
	\end{list4}

	For the optimization, we assume that all nodes are aware of 
	the network diameter $D$ (or an upper bound on $D$), an upper bound on time-steps $1+\bar{\tau}\leq B$,
	the augmented Lagrangian function parameter $\rho$, and the ADMM maximum  optimization step $k_{\max}$\footnote{Note that the ADMM stopping criterion is the primal and dual feasibility condition in~\cite{boyd2011distributed}.}.

	\begin{algorithm}[t!]
		\caption{AsyAD-ADMM: Asynchronous Approximate Distributed ADMM}
		\begin{algorithmic}[1]
			\STATE \textbf{Input:} A strongly connected digraph $\mathcal{G}(\mathcal{V}, \mathcal{E})$, in which each node  $v_{i} \in \mathcal{V}$ knows its out-degree $\mathcal{N}_i^{+}$, $\rho >0$, network diameter $D$ (or an upper bound), upper bound on the number of time-steps $\bar{\tau}$, error tolerance $ \epsilon $, $ k_{\max} $ (ADMM maximum number of iterations)
			\STATE \textbf{Initialization:} At $k=0$, each node $ v_i \in \mathcal{V}$ sets $ x_i^0$, $z_i^0$, and $\lambda_i^0$ randomly.
			\STATE Node $ v_i \in \mathcal{V}$ does the following:
			\WHILE {$ k \le k_{\max} $ }
			\STATE Compute $x_i^{k+1}$ using Eq.~\eqref{dadmm_x}
			\STATE Compute $z_i^{k+1}$ via Algorithm~\ref{Algorithm_finitetimeRCwDelays} by setting $ y_i^0 = x_i^{k+1} + \lambda_i^{k}/\rho $ and using $ D$, $\bar{\tau}$, and $\epsilon $.
			\STATE Compute $\lambda_i^{k+1}$ using Eq.~\eqref{dadmm_lamda}
			\IF{ADMM stopping criterion is satisfied}
			\STATE Stop AsyAD-ADMM
			\ENDIF
			\STATE $k\leftarrow k+1$ 
			\ENDWHILE
		\end{algorithmic}
		\label{algorithm:1}
	\end{algorithm}

	\subsection{Algorithm~\ref{Algorithm_finitetimeRCwDelays}}
	\label{sec:distributed_synchronous}

		Under Assumption \ref{assup_graph}, Cady \emph{et al.} in \cite{2015:Cady} proposed an algorithm which is based on the synchronous ratio-consensus protocol \cite{2010:christoforos} and takes advantage of synchronous $\min$- and $\max$-consensus iterations to allow the nodes to determine the time step, $k_0$, when their ratios,  
	(i.e., $\{z_{i}^{k_0} | v_{i}\in \mathcal{V} \}$ in ADMM), are within $\epsilon$ of each other. However, the algorithm in \cite{2010:christoforos} cannot deal with asynchronous information exchange scenarios. To circumvent this problem, and inspired by Cady \emph{et al.} in \cite{2015:Cady}, the authors in \cite{Salapaka:2020} adopted robustified ratio consensus proposed in \cite{2014:chrisThemisTAC} to propose a termination mechanism for average consensus with delays. Similarly, \cite{grammenos2021cpu} proposed the corresponding asynchronous termination algorithm in the context of a quadratic distributed optimization problem, in which the algorithm is executed only once. We adopt the asynchronous termination algorithm proposed in \cite{grammenos2021cpu} to compute  $z_i^{k+1}$, and we extend it to accommodate the repetition of this algorithm. The algorithm, under Assumption~\ref{assump_time_index_bound} is summarized in Algorithm~\ref{Algorithm_finitetimeRCwDelays}. Specifically, Algorithm~\ref{Algorithm_finitetimeRCwDelays} makes use of the following ideas:
	\begin{list4}
		\item Each node $v_{j}$ runs asynchronous ratio consensus in Sec.~\ref{arc}; in our case, we use initial conditions $y_{j}^0=x_j^{k+1} + \lambda_j^{k}/\rho$.
		\item At the same time, each node maintains two auxiliary states, $m_j^k$ and $M_j^k$, which are updated using asynchronous $\min$- and $\max$-consensus in Sec.~\ref{amc}, respectively. Both converge in $(1+\bar{\tau})D$  steps \cite{2013:Giannini}. 
		\item Every $(1+\bar{\tau})D$ steps each node checks whether $\|M_j^k-m_j^k\|<\epsilon$. If this is the case, then the ratios for all nodes are close to the asymptotic value and it stops iterating. Otherwise, $m_j^k$ and $M_j^k$ are reinitialized to {$z_{j} ^k$}.
	\end{list4}

	\begin{algorithm}[t!]
		\caption{Distributed Finite-Time Termination for Asynchronous Ratio Consensus}
		\begin{algorithmic}[1]
			\STATE \textbf{Input:} A strongly connected digraph $\mathcal{G}=(\mathcal{V}, \mathcal{E})$.   $v_{j} \in \mathcal{V}$ knows its out-degree $\mathcal{N}_j^{+}$.
			$y_{j}^0$, $w_{j}^0=\mathbb{1}$, $ D, \bar{\tau}, \epsilon $.
			\STATE \textbf{set} $M_j^0=+\infty, ~ m_j^0=-\infty,  {\rm flag}_j^0=0, z_{j}= \frac{y_{j}^0}{w_{j}^0}$
			\STATE \textbf{set} $p_{lj}=\frac{1}{1+\mathcal{D}_j^+}$, $\forall~v_l\in \mathcal{N}_j^{+} \cup \{v_{j}\}$ (zero otherwise)
			\FOR{$k \geq 0$}
			\WHILE{${\rm flag}_j^k=0$}
			\IF{$k \mod (1+\bar{\tau})D =0$ and $k\neq 0$}
			\IF{$\|M_j^k-m_j^k\|< \epsilon$}
			\STATE \textbf{set} ${\rm flag}_j^k = 1$
			\ENDIF
			\STATE \textbf{set} $M_j^k=m_j^k=z_{j}^k = \frac{y_{j}^k}{w_{j}^k}$
			\ENDIF
			\STATE  \textbf{broadcast} to all $v_l \in \mathcal{N}_j^{+}$: \newline $p_{lj}y_{j}^k$, $p_{lj}w_{j}^k$, $M_j^k$, $m_j^k$
			\STATE  \textbf{receive} from all $v_{i} \in \mathcal{N}_j^{-}[k]$
			: \newline $p_{ji}y_{i}^k$, $p_{ji}w_i^k$, $M_i^k$, $m_i^k$
			\STATE  \textbf{compute} \\
			\STATE $y_{j}^k\hspace{-0.1cm}\leftarrow \hspace{-0.1cm} p_{jj}y_{j}^k + \sum_{v_{i} \in \mathcal{N}^{-}_j} \sum_{r=0}^{\bar{\tau} } y_{ji}^{k-r}I_{k-r,ji}^{r}$
			\STATE $w_{j}^k\hspace{-0.1cm}\leftarrow \hspace{-0.1cm} p_{jj}w_{j}^k + \sum_{v_{i} \in \mathcal{N}^{-}_j} \sum_{r=0}^{\bar{\tau} } w_{ji}^{k-r}I_{k-r,ji}^{r}$
			\STATE $M_j^k\hspace{-0.1cm}\leftarrow \hspace{-0.1cm} \max_{v_{i}\in \mathcal{N}_j^{-} \cup \{v_{j}\}}\{ M_i[t_j(k)+\theta_{ij}(k)] \}$
			\STATE $m_j^k\hspace{-0.1cm}\leftarrow \hspace{-0.1cm} \max_{v_{i}\in \mathcal{N}_j^{-} \cup \{v_{j}\}}\{ m_i[t_j(k)+\theta_{ij}(k)] \}$
			\ENDWHILE
			\ENDFOR
		\end{algorithmic}
		\label{Algorithm_finitetimeRCwDelays}
	\end{algorithm}
	
As stated in Assumption~\ref{assump_time_bound}, the upper bound $ B $ on time-steps is guaranteed to be executed within $ T $ seconds (actual time). So, nodes have a more loose synchronization of every $ DT $ seconds. The $ \max- $ and $ \min- $ are checked every $ DT $ seconds and if they have converged in one of the checks, the next optimization step is initiated for all $ DT $ seconds, after the termination of the algorithm. This approach requires some loose form of synchronization, but the time scale is much more coarse.

	\subsection{Convergence analysis}\label{sec:convergence_analysis}
	The following Theorem~\ref{theorem:main} states that AsyAD-ADMM has $ O(\frac{1}{k}) $ convergence rate.
	\begin{theorem}[Convergence]\label{theorem:main}
		Let $ \{X^k, z^k, \lambda^k\} $ be the iterates in AsyAD-ADMM algorithm for problem~\eqref{objective_function}, where $ X^k = [{x_1^{k}}\T,{x_2^{k}}\T, \ldots, {x_n^{k}}\T]\T $ and $ \lambda^k = [{\lambda_1^{k}}\T,{\lambda_2^{k}}\T, \ldots, {\lambda_n^{k}}\T]\T $. Let $ \bar{X}^k = \frac{1}{k} \sum_{s=0}^{k-1}X^{s+1}, \bar{z}^k = \frac{1}{k} \sum_{s=0}^{k-1}z^{s+1} $ be respectively the ergodic average of $ X^k $.
		Considering a strongly connected communication graph, 
		under Assumptions \ref{assup_graph}-\ref{assump_time_bound}, the following relationship holds for any iteration $ k $ as
		\begin{align}
		0&\le L(\bar X^{k},\bar z^{k}, \lambda^{*})- L ( X^{*},z^{*},\lambda^{*}) \label{convergence_relationship}\\
		&\le \frac{1}{k}\left( \frac{1}{2\rho} \|\lambda^{*}-\lambda^{0}\|^2 + \frac{\rho}{2}\|X^{*}-z^{0}\|^2\right) + \mathcal{O}(\sqrt{n}\epsilon),\nonumber
		\end{align}
		where $ \epsilon $ is the $ z $ update~\eqref{dadmm_z} tolerance whose value is independent of the researched problem.
	\end{theorem}
	
	\begin{proof}
		See Appendix~\ref{proof:theorem:main}.
	\end{proof}
		\begin{remark}\label{remark_n_epsilon}
			Note that the convergence proof shows the convergence rate for the optimization steps. 
			The fact that the distributed algorithm is approximate, an additional term $ \mathcal{O}(\sqrt{n}\epsilon)$ is introduced, that it is a function of the error tolerance and the size of the network and that it influences the solution precision. 
	\end{remark}
	\begin{remark}\label{remark_comparison}
		The convergence proof here is basically different from the ones in~\cite{wei2012distributed} and \cite{khatana2020d} as the investigated problems are different. Specifically, in~\cite{wei2012distributed}, the proposed D-ADMM can be only applied to nodes with undirected graphs as the constraint $ AX=0 $ is needed to minimize the objective function~\eqref{problem_reformulated}, where the matrix $ A $ is related to the communication graph structure which must be undirected. Authors in \cite{khatana2020d} proposed the D-ADMM for digraphs with the same constraint $ \| x_i-x_j \|\le \epsilon, \epsilon>0 $. However, it can only be applied to synchronous case.
	\end{remark}
	\section{Numerical Example}\label{sec:examples}

	The distributed least square problem is considered as
	\begin{equation}\label{problem_lesatSquare}
	\operatorname*{argmin}_{x\in \mathbb{R}^{p}}  f(x) = \frac{1}{2}\sum_{i=1}^{n} \|A_ix-b_i\|^2,
	\end{equation}
	where $ A_i  \in \mathbb{R}^{q\times p} $ is only known to node $ v_i $, $ b_i \in \mathbb{R}^{q} $ is the measured data and $ x \in \mathbb{R}^{ p} $ is the common decision variable that needs to be optimized. For the automatic generation of large number of different matrices $ A_i $, we choose $ q = p $ to have the square $ A_i $. All elements of $ A_i $ and $ b_i $ are set from independent and identically distributed (i.i.d.) samples of standard normal distribution $ \mathcal{N}(0,1) $.

	We choose $ p=3 $ and  $ n=600 $ to have 600 nodes having a strongly connected digraph. We implement the D-ADMM-FTERC algorithm in~\cite{wei_themis_2021} as a benchmark to evaluate factors  {(i.e., $ \epsilon, \bar \tau $)} influencing AsyAD-ADMM performance as D-ADMM-FTERC is developed for calculating the optimal solution to distributed optimization problems for digraphs, i.e., all figures related to D-ADMM-FTERC are optimal.
	 However, D-ADMM-FTERC is synchronous and this is the reason we develop AsyAD-ADMM in this paper.
	
	 Fig.~\ref{fig_obj_comparison} demonstrates three points: (i) AsyAD-ADMM solutions are very close to the optimal one from D-ADMM-FTERC. (ii) From Fig.~\ref{fig_obj_comparison} (a) to (b), the smaller the value of $ \epsilon $ is, the closer AsyAD-ADMM solutions are to D-ADMM-FTERC one. This is reasonable as smaller value of $ \epsilon $ means better and more accurate ``approximate'' ratio consensus for $ z $ update~\eqref{dadmm_z} using Algorithm~\ref{Algorithm_finitetimeRCwDelays}. 
	 \begin{figure}[t]
	 \centering
		\includegraphics[width=0.9\columnwidth]{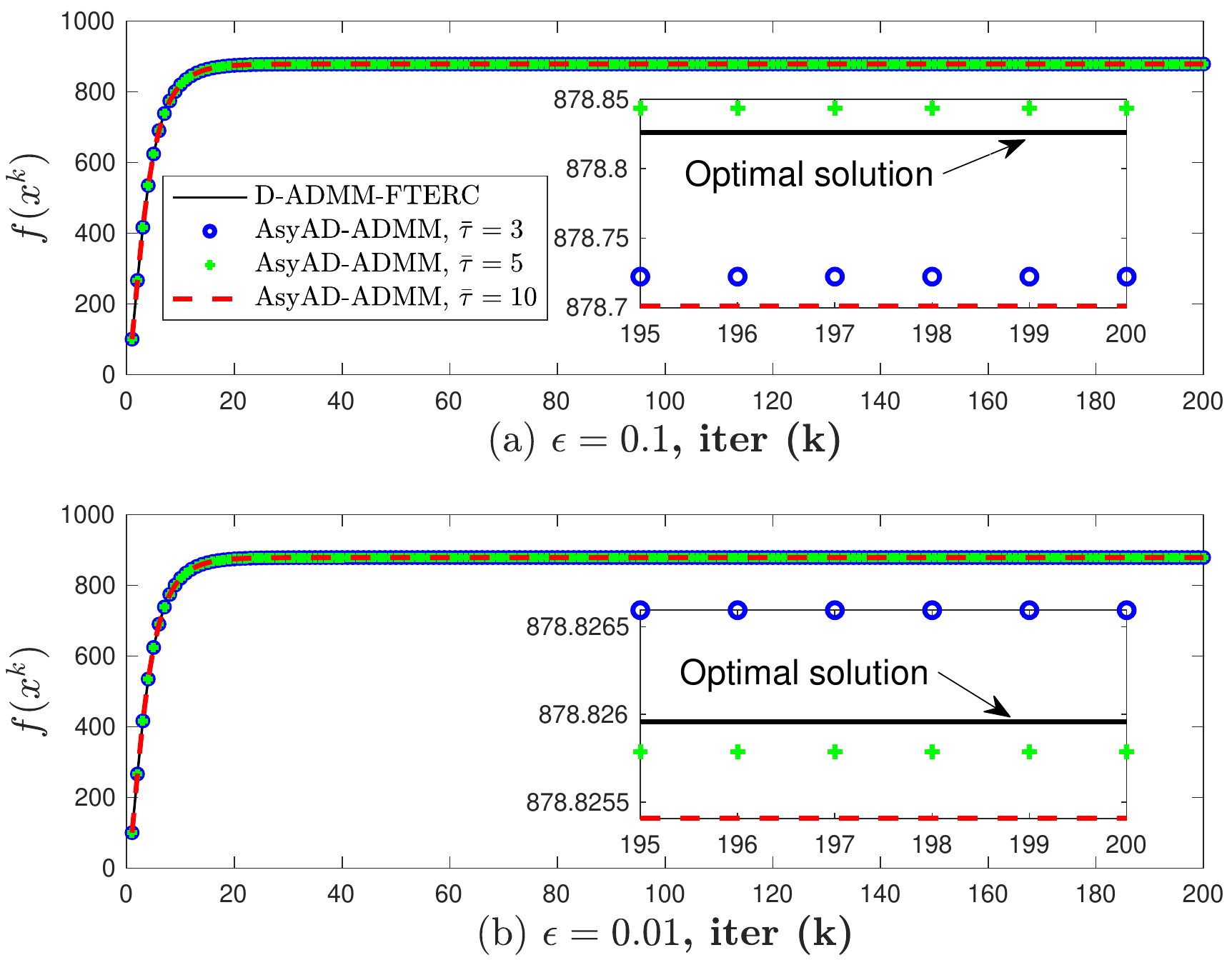}
		\caption{Comparison of solutions  to problem~\eqref{problem_lesatSquare} between AsyAD-ADMM  in this paper and synchronous D-ADMM-FTERC in~\cite{wei_themis_2021} without stopping condition
			 (steps 8-10 in Algorithm~\ref{algorithm:1},   $ k_{\max} = 200 $): subplots from (a) to (b) are
		related to different values of $ \epsilon $ and $ \bar \tau $ influencing the performance of AsyAD-ADMM.
	}
		\label{fig_obj_comparison}
	\end{figure}
	 This can be also validated in Fig.~\ref{fig_z_rcStep_comparison} where red and black lines are the $ z $ updates using AsyAD-ADMM and D-ADMM-FTERC, respectively. (iii) The values of $ \bar \tau $ (defined in Sec.~\ref{arc}) do not influence the precision of AsyAD-ADMM, which also makes sense as $ \bar \tau $ decides the degree of asynchrony, not the asynchronous ADMM precision. Nodes use this information to decide when to update as shown in step 6 of Algorithm~\ref{Algorithm_finitetimeRCwDelays}. The values of  $ \bar \tau $ do influence the iteration steps of Algorithm~\ref{Algorithm_finitetimeRCwDelays} as described in Table~\ref{tab_stepComparison} from which it shows the larger the value of $ \bar \tau $ is, the more iteration steps Algorithm~\ref{Algorithm_finitetimeRCwDelays} needs in each AsyAD-ADMM  optimization step.
	  For $ \epsilon = 0.01 $, the steps  are the same as $ 1000 $ because we set the iteration step upper bound for  Algorithm~\ref{Algorithm_finitetimeRCwDelays} as 1000; it means iteration steps are no less than 1000 when $ \epsilon \le 0.01 $ and $ \bar \tau \ge 3 $.
	 	\begin{figure}[t]
	 	\centering
		\includegraphics[width=0.9\columnwidth]{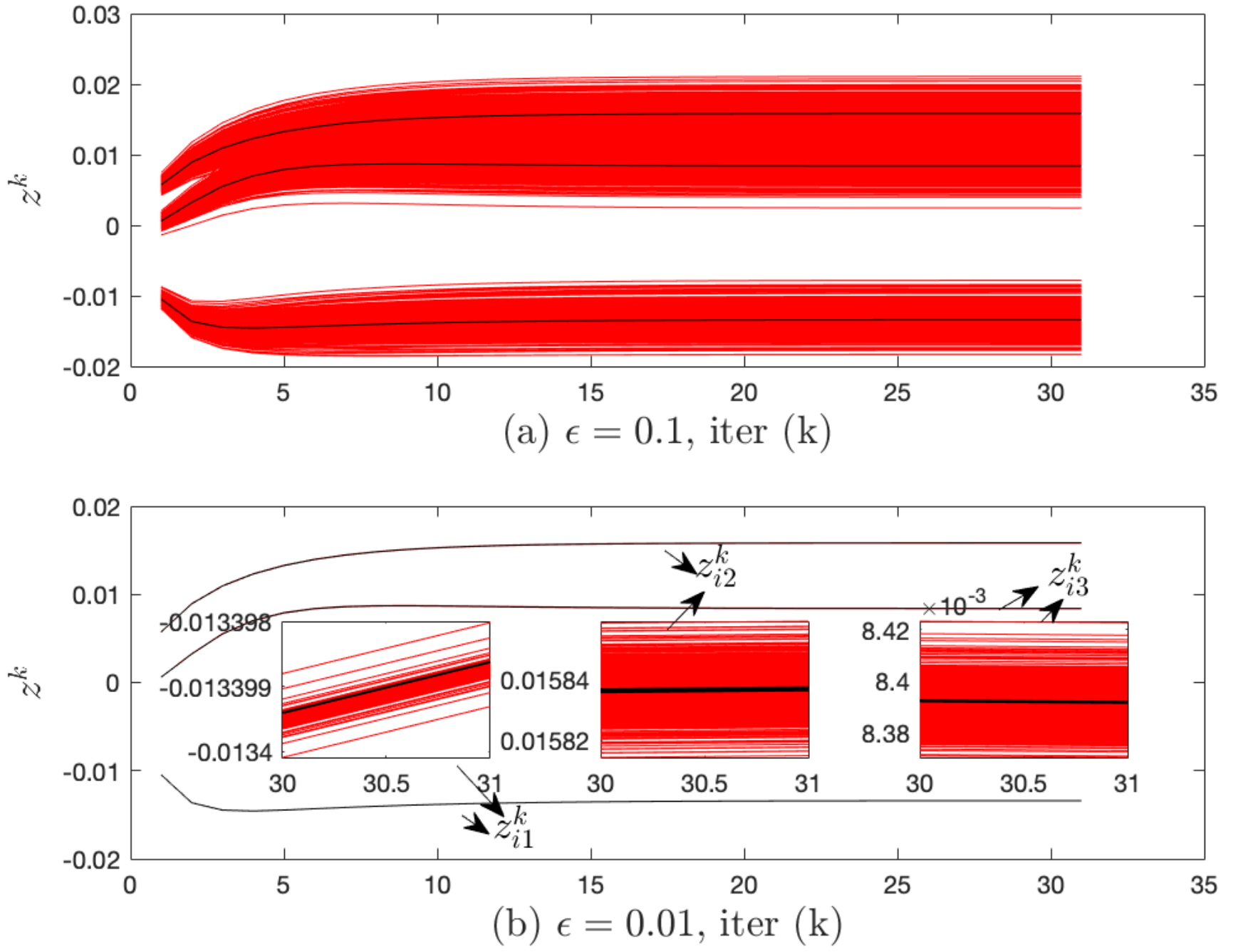}
		\caption{$ z^{k} $ ($ z^{k}_{ij}, i = 1,\ldots,n, j= 1,\ldots,p $) update in~\eqref{dadmm_z} with $ \bar \tau = 3 $ (red and black lines are from AsyAD-ADMM and D-ADMM-FTERC, respectively ) under stopping condition
			(steps 8-10 in Algorithm~\ref{algorithm:1},  absolute tolerance:  $  1\exp^{-4}$,
			relative tolerance:   $1\exp^{-2}$).
}
		\label{fig_z_rcStep_comparison}
	\end{figure}
	 Note that in Fig.~\ref{fig_z_rcStep_comparison}, we have AsyAD-ADMM with the stopping condition which is for the convenience of figure presentation.
	 The reason  that  Fig.~\ref{fig_z_rcStep_comparison} (a) and (b) have a big difference is $ \|z_i-z_j\|\le \epsilon $, which means a smaller $ \epsilon $ leads to a more accurate $ z $ update~\eqref{dadmm_z} using Algorithm~\ref{Algorithm_finitetimeRCwDelays}. And the reason is that a smaller $ \epsilon $ leads to a larger Algorithm~\ref{Algorithm_finitetimeRCwDelays} iteration step as one can check it is $ 9 $ and $ 1000 $  from Table~\ref{tab_stepComparison} for $\epsilon = 0.1 $  and $\epsilon = 0.01 $ in this example.
	Also from Fig.~\ref{fig_z_rcStep_comparison}, one can see the ADMM iteration numbers are the same for different values of $ \epsilon $ for AsyAD-ADMM with the stopping condition.
	 Furthermore, Table~\ref{tab_timeComparison} describes the running time comparison on an Intel Core i5 processor at 2.6 GHz with Matlab R2020b for different values of $ \epsilon $ and $ \bar \tau $ with $ k_{\max}=200 $.
	 One can see for $ \epsilon=0.01 $, from $ \bar \tau = 3, 5, 10 $, even though Algorithm~\ref{Algorithm_finitetimeRCwDelays} iteration steps are the same as 1000, the running time is larger and larger. 
	 \begin{table}[h]
		\centering
		\caption{Algorithm~\ref{Algorithm_finitetimeRCwDelays} iteration step in each Algorithm~\ref{algorithm:1} iteration.}
		\label{tab_stepComparison}
		{\small
			\begin{tabular}{clll}
				\hline
				\multicolumn{1}{|c|}{}             & \multicolumn{1}{c|}{$ \bar \tau = 3 $}       & \multicolumn{1}{c|}{$ \bar \tau = 5 $}      & \multicolumn{1}{l|}{$ \bar \tau =10 $}        \\ \hline
				\multicolumn{1}{|c|}{$ \epsilon = 0.1 $}   & \multicolumn{1}{l|}{9} & \multicolumn{1}{l|}{13} & \multicolumn{1}{l|}{23} \\ \hline
				\multicolumn{1}{|c|}{$ \epsilon = 0.01 $}   & \multicolumn{1}{l|}{1000} & \multicolumn{1}{l|}{1000} & \multicolumn{1}{l|}{1000} \\ \hline
				&                                 &                                 &                                  
		\end{tabular}}
	\end{table}
	 	\begin{table}[h]
		\centering
		\caption{AsyAD-ADMM Algorithm~\ref{algorithm:1} running time.}
		\label{tab_timeComparison}
		{\small
			\begin{tabular}{clll}
				\hline
				\multicolumn{1}{|c|}{}             & \multicolumn{1}{c|}{$ \bar \tau = 3 $}       & \multicolumn{1}{c|}{$ \bar \tau = 5 $}      & \multicolumn{1}{l|}{$ \bar \tau =10 $}        \\ \hline
				\multicolumn{1}{|c|}{$ \epsilon = 0.1 $}   & \multicolumn{1}{l|}{7.9677s} & \multicolumn{1}{l|}{12.0282
					s} & \multicolumn{1}{l|}{33.6401s} \\ \hline
				\multicolumn{1}{|c|}{$ \epsilon = 0.01 $}   & \multicolumn{1}{l|}{190.1420s} & \multicolumn{1}{l|}{246.9963s} & \multicolumn{1}{l|}{470.6576s} \\ \hline
				&                                 &                                 &                                  
		\end{tabular}}
	\end{table}
	 This is because Algorithm~\ref{Algorithm_finitetimeRCwDelays} running time is  the time gap between two consecutive iterations multiplies the iteration step while in step 6 of Algorithm~\ref{Algorithm_finitetimeRCwDelays}, $ (1+\bar{\tau})D $ is the time gap. 
	 Fig.~\ref{fig_X_Xstar} demonstrates two points. (i) It validates Theorem~\ref{theorem:main} that AsyAD-ADMM has $\mathcal{O}(\frac{1}{k}) $ convergence rate. (ii) It is in accordance with Fig.~\ref{fig_obj_comparison} and Fig.~\ref{fig_z_rcStep_comparison} that a smaller value of $ \epsilon $ has a better solution precision.
	 
	 What is more, from  Fig.~\ref{fig_obj_comparison} and Table~\ref{tab_timeComparison}, considering the AsyAD-ADMM solution precision and running time, it is better to have the value of $ \epsilon $ neither too large nor too small; for this example, $ \epsilon \in [0.01, 0.1)$ is a good value range.

	\begin{figure}[ht]
	\centering
		\includegraphics[width=0.9\columnwidth]{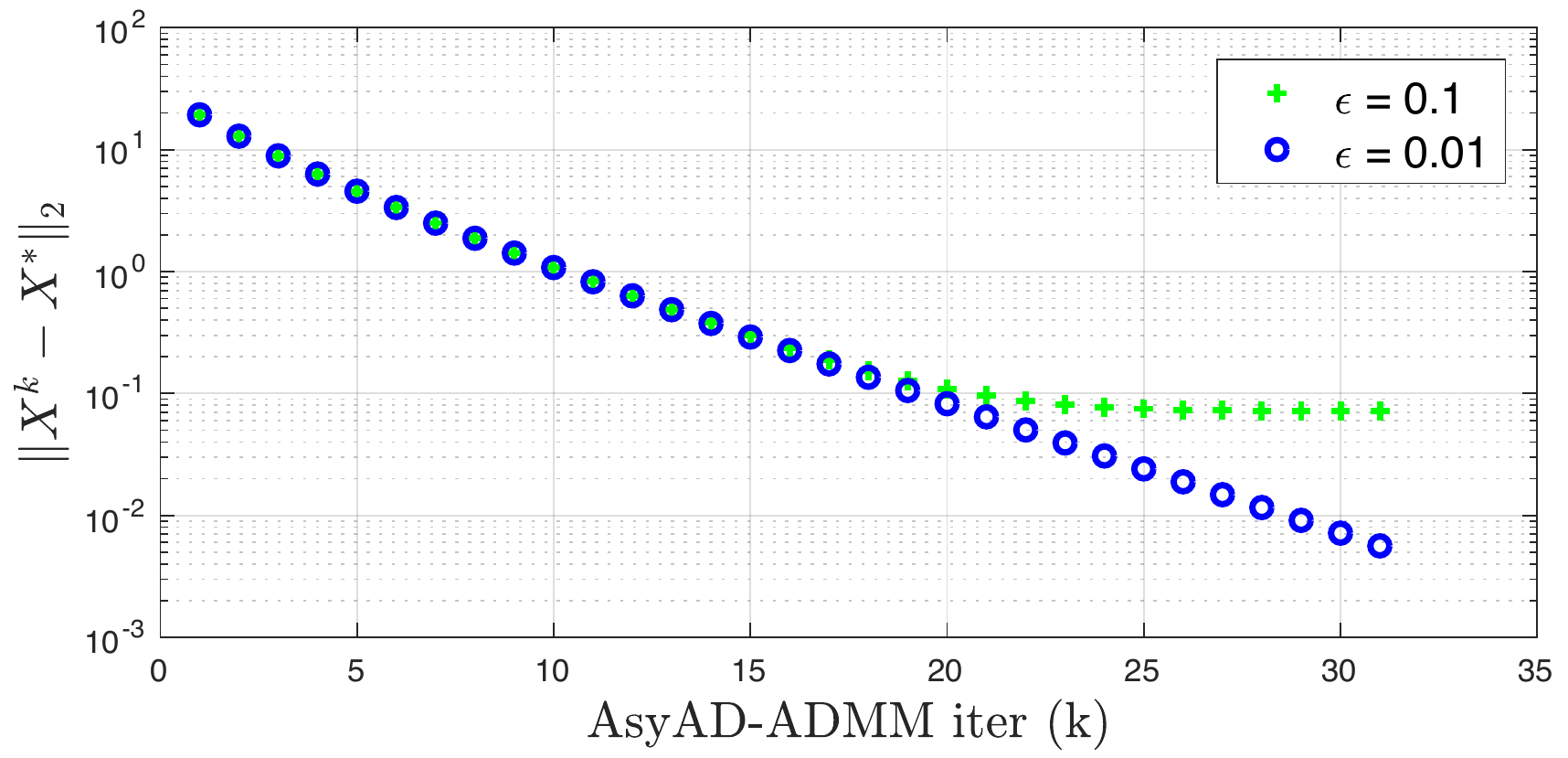}
		\caption{Comparison of the absolute errors against different values of $ \epsilon $ with $  \bar \tau = 3 $ and the stopping condition
			(steps 8-10 in Algorithm~\ref{algorithm:1},  absolute tolerance:  $  1\exp^{-4}$,
			relative tolerance:   $ 1\exp^{-2} $). 
		}
		\label{fig_X_Xstar}
	\end{figure}

	\section{Conclusions and Future Directions}\label{sec:conclusions}
	
	An asynchronous approximate distributed alternating direction method of multipliers algorithm is proposed to provide a solution for the distributed optimization problems allowing asynchronous information exchange among nodes in digraphs with assumptions of bounded time-index steps of asynchronous nodes and the known digraph diameter to all nodes.
	By proposing  the finite-time asynchronous approximate ratio consensus algorithm for one ADMM primal variable update, the solution is close to the optimal but acceptable.
	How to choosing algorithm parameters to have an aggressive performance is also discussed.
	
	Future work will focus on
	 designing a new asynchronous ADMM version without the communication graph diameter information
	  to compute the optimal solution for distributed optimization problems.

	\appendices
	
	\section{Proof of Theorem~\ref{theorem:main}}
	\label{proof:theorem:main}
	
	The analysis is inspired by \cite{wei2012distributed} and \cite{khatana2020d}. 
	From the second inequality of the saddle point of Lagrangian function~\eqref{saddle_point}, the first inequality in~\eqref{convergence_relationship} can be proved directly.
	
	We now prove the second inequality in~\eqref{convergence_relationship}. For each node $ v_i $, since $ x_i^{k+1} $ minimizes $ L_{\rho}(x,z^k, \lambda^k) $ in~\eqref{dadmm_x}, by the optimal condition, we have 
	\begin{equation}
	\begin{aligned}
	(x- x_i^{k+1})\T[h_i(x_i^{k+1}) + \lambda_i^k + \rho (x_i^{k+1} - z_i^k) ] \ge 0,
	\end{aligned}
	\end{equation}
	where $ h_i(x_i^{k+1}) $ is the sub-gradient of $ f_i $ at $ x_i^{k+1} $. By integrating $ x_i^{k+1} = (\lambda_i^{k+1}-\lambda_i^k)/\rho + z_i^{k+1} $ from~\eqref{dadmm_lamda} into the above inequality, we have
	\begin{equation}
	(x- x_i^{k+1})\T[h_i(x_i^{k+1}) + \lambda_i^{k+1} + \rho (z_i^{k+1} - z_i^k) ] \ge 0.
	\end{equation}
	The  above $ n $ inequalities can be written in compact form as
	\begin{equation}\label{optimal_condition_X}
	(X- X^{k+1})\T[\bar h(X^{k+1}) + \lambda^{k+1} + \rho (z^{k+1} - z^k) ] \ge 0,
	\end{equation}
	where $ \bar h(X^{k+1}) = [h_1\T(x_1^{k+1}), \ldots, h_n\T(x_n^{k+1})]\T $.
	Denote
	\begin{equation}\label{z_kPlus1_optimal}
	\tilde z_i^{k+1} \coloneqq  \frac{1}{n} \sum_{i=1}^{n}(x_i^{k+1} + \frac{\lambda_i^{k}}{\rho}  ),
	\end{equation}
	and $ \tilde z^{k+1} =  [\tilde z_1^{{(k+1)}\T}, \ldots, \tilde z_n^{{(k+1)}\T}]\T $.
	Based on ``approximate'' consensus Algorithm~\ref{Algorithm_finitetimeRCwDelays}, we denote the real $ z $ update~\eqref{dadmm_z} for the step $ k+1 $ as $ z^{k+1} \coloneqq \tilde z^{k+1} + e^{k+1},  e^{k+1} =  [ e_1^{{(k+1)}\T}, \ldots,  e_n^{{(k+1)}\T}]\T $ with $ \|e_i^{k+1}\| \le \epsilon, i = 1, \ldots, n, k=0, 1, \ldots $, i.e., $ \|e^{k+1}\|_2 \le \sqrt{n} \epsilon $.
	Since $ \tilde z^{k+1} $ minimizes $ L_{\rho}(x^{k+1},z, \lambda^k) $ in~\eqref{dadmm_z}, similarly, we have
	\begin{equation}\label{optimal_condition_z}
	\begin{aligned}
	(z- &\tilde z^{k+1})\T[\bar g(\tilde z^{k+1}) - \lambda^{k} - \rho (X^{k+1} -\tilde z^{k+1}) ]\\
	&=(z-\tilde z^{k+1})\T (\bar g(\tilde z^{k+1}) - \lambda^{k+1} -\rho e^{k+1} ) \ge 0
	\end{aligned}
	\end{equation}
	for all $ z \in \mathcal{C} $,
	where $ \bar g(\tilde z^{k+1}) $ is the sub-gradient of $ g $ at $ \tilde z^{k+1} $. As both $ F $ and $ g $ are convex, by utilizing the sub-gradient inequality, from \eqref{optimal_condition_X} and \eqref{optimal_condition_z}, we get 
	\begin{align}
	F(&X^{k+1}) - F(X) + g(\tilde z^{k+1}) - g(z) \nonumber\\
	\le& -(X- X^{k+1})\T\bar h(X^{k+1}) - (z- \tilde z^{k+1})\T\bar g(\tilde z^{k+1})\nonumber \\
	\le& 	(X- X^{k+1})\T[ \lambda^{k+1} + \rho (z^{k+1} - z^k) ]\nonumber\\
	& + (z-\tilde z^{k+1})\T ( - \lambda^{k+1} -\rho e^{k+1} )\nonumber \\
	=& \lambda^{{(k+1)}\T} (X- X^{k+1} -z+  z^{k+1}-e^{k+1}) \nonumber\\
	&+ \rho (X- X^{k+1})\T(z^{k+1} - z^k) -\rho (z-\tilde z^{k+1})\T   e^{k+1} \nonumber \\
	=& \lambda^{{(k+1)}\T} [X- X^{k+1} -z+  z^{k+1}] + \rho (X- X^{k+1})\T(z^{k+1} \nonumber\\
	&- z^k) -[\lambda^{k+1}+ \rho (z-\tilde z^{k+1})]\T   e^{k+1}. \label{proof1}
	\end{align}
	From \eqref{setC}, \eqref{indicator_function} and \eqref{z_kPlus1_optimal}, we have $ g(\tilde z^{k+1}) = g( z^{k+1}) $ as $ \{\tilde z^{k+1}, z^{k+1} \} \in \mathcal{C}$.
	Due to feasibility of the optimal solution $ (X^{*}, z^{*}) $, we obtain $ X^{*}- z^{*}=0 $.
	By setting $ X=X^{*}, z=z^{*}  $, \eqref{proof1} becomes
	\begin{align}
	F&(X^{k+1}) - F(X^{*}) + g(z^{k+1}) - g(z^{*})
	\le \lambda^{{(k+1)}\T} \label{proof2}\\& \times(z^{k+1}- X^{k+1} ) + \rho (X^{*}- X^{k+1})\T(z^{k+1} - z^k) -\eta^{k+1}. \nonumber
	\end{align}
	where $ \eta^{k+1} \coloneqq  [\lambda^{k+1}+ \rho (z^{*}-\tilde z^{k+1})]\T   e^{k+1} $.
	Note that in Appendix A of \cite{boyd2011distributed}, we know that both $ \lambda^{k+1} $ and $ z^{*}-\tilde z^{k+1} $ are bounded; and $ z^{*}-\tilde z^{k+1} $ goes to zero as $ k \rightarrow \infty $.
	Therefore, for $ k=1, 2, \ldots $, there exist constant numbers $ M_{\lambda} $ and $ M_{z} $ such that 
	\begin{align*}
	\|\lambda^{k}\| \le M_{\lambda}, \|z^{*}-\tilde z^{k}\| \le  M_{z}^{k}, M_{z}^{k} \le M_{z}, \lim_{k\to \infty}  M_{z}^{k} =0.
	\end{align*}
	Denote $ \bar F^{k+1} \coloneqq F(X^{k+1}) - F(X^{*}) + g(z^{k+1}) - g(z^{*}) + {\lambda^{*}}\T( X^{k+1}-z^{k+1}) $ and $ \phi^{k+1} \coloneqq \rho (X^{*}- X^{k+1})\T(z^{k+1} - z^k)$.
	Adding $ {\lambda^{*}}\T( X^{k+1}-z^{k+1} ) $ to both sides of \eqref{proof2} yields
	\begin{align}
	\bar F^{k+1}
	\le& (\lambda^{*}-\lambda^{k+1})\T (X^{k+1}-z^{k+1} ) + \phi^{k+1}- \eta^{k+1} \nonumber
	\\
	\le&\frac{1}{\rho}(\lambda^{*}-\lambda^{k+1})\T(\lambda^{k+1}-\lambda^{k} ) + \phi^{k+1} - \eta^{k+1},\label{proof3}
	\end{align}
	where the last equality is calculated from~\eqref{dadmm_lamda}. Recall an well-known equality law that 
	\begin{align}
	(a_1&-a_2)\T(a_3-a_4) =\frac{1}{2}(\|a_1-a_4\|^2- \|a_1-a_3\|^2 ) \label{equality_law}\\
	&+\frac{1}{2}(\|a_2-a_3\|^2- \|a_2-a_4\|^2 ), \forall a_1,a_2,a_3,a_4 \in \mathbb{R}^p.\nonumber
	\end{align}
	Then, by using equality~\eqref{equality_law}, \eqref{proof3} changes to
	\begin{align}
	\bar F^{k+1}
	\le&\frac{1}{2\rho}(\|\lambda^{*}-\lambda^{k}\|^2- \|\lambda^{*}-\lambda^{k+1}\|^2 
	- \|\lambda^{k+1}-\lambda^{k}\|^2 )	\nonumber\\& + \frac{\rho}{2} (\|X^{*}-z^k\|^2- \|X^{*}-z^{k+1}\|^2\nonumber\\& +\| X^{k+1}-z^{k+1}\|^2- \| X^{k+1}-z^k\|^2 ) - \eta^{k+1}\nonumber\\
	\le & \frac{1}{2\rho}(\|\lambda^{*}-\lambda^{k}\|^2- \|\lambda^{*}-\lambda^{k+1}\|^2 )\label{proof4}\\&+\frac{\rho}{2} (\|X^{*}-z^k\|^2- \|X^{*}-z^{k+1}\|^2) - \eta^{k+1},\nonumber
	\end{align}
	where the last inequality comes from using~\eqref{dadmm_lamda} and dropping the negative term $ -\frac{\rho}{2}\| X^{k+1}-z^k\|^2 $. Now, by using $ s \coloneqq k $, we change~\eqref{proof4} to another format as
	\begin{align}
	\bar F^{s+1}
	\le & \frac{1}{2\rho}(\|\lambda^{*}-\lambda^{s}\|^2- \|\lambda^{*}-\lambda^{s+1}\|^2 )\nonumber\\&+\frac{\rho}{2} (\|X^{*}-z^s\|^2- \|X^{*}-z^{s+1}\|^2) -  \eta^{s+1},\label{proof5}
	\end{align}
	which holds true for all $ s $. 
	Denote the constant $ \theta \coloneqq \|\lambda^{*}-\lambda^{0}\|^2/(2\rho)+\rho \|X^{*}-z^0\|^2/2 $.
	By summing~\eqref{proof5} over $ s=0,1,\ldots, k-1 $ and after telescoping calculation, we have
	\begin{align}
	\sum_{s=0}^{k-1}F(&X^{s+1}) - kF(X^{*}) + \sum_{s=0}^{k-1}g(z^{s+1}) - kg(z^{*}) \nonumber\\&+ {\lambda^{*}}\T\sum_{s=0}^{k-1}( X^{s+1}-z^{s+1}) \label{proof6}\\
	\le & \theta - \frac{1}{2\rho} \|\lambda^{*}-\lambda^{k}\|^2 -\frac{\rho}{2}  \|X^{*}-z^{k}\|^2 -\sum_{s=0}^{k-1} \eta^{s+1}.\nonumber
	\end{align}
	Due to the convexity of both $ F $ and $ g $, we get $ kF(\bar X^k) \le \sum_{s=0}^{k-1}F(X^{s+1}) $ and $ kg(\bar z^k) \le \sum_{s=0}^{k-1}g(z^{s+1}) $. Thus, by the definition of $ \bar X^k, \bar z^k $ and dropping the negative terms,
	\begin{align}
	kF&(\bar X^k) - kF(X^{*}) + kg(\bar z^k) - kg(z^{*}) + {\lambda^{*}}\T( k\bar X^k-k\bar z^k)\nonumber\\
	\le & \theta -\sum_{s=0}^{k-1}\eta^{s+1} 
	\le  \theta + k \underbrace{(M_{\lambda}+ \rho M_{z})\sqrt{n}\epsilon}_{\coloneqq \mathcal{O}(\sqrt{n}\epsilon)}.\label{proof7}
	\end{align}
	Based on $ X^{*}- z^{*}=0 $, \eqref{proof7} combined with the definition of Lagrangian function [cf.~\eqref{Lagrangian}] prove~\eqref{convergence_relationship}.

	\bibliographystyle{IEEEtran}
	\bibliography{bib_finite}
	
\end{document}